\documentclass{article}

\usepackage{arxiv}

\usepackage[utf8]{inputenc} 
\usepackage[T1]{fontenc}    
\usepackage{hyperref}       
\usepackage{url}            
\usepackage{booktabs}       
\usepackage{amsfonts}       
\usepackage{nicefrac}       
\usepackage{microtype}      
\usepackage{lipsum}		
\usepackage{graphicx}
\usepackage{natbib}
\usepackage{doi}
\usepackage{amsmath,amssymb,amsthm}
\usepackage{geometry}
\usepackage{hyperref}
\usepackage{enumitem}
\usepackage{subcaption}
\usepackage{color}
\usepackage{float}

\newtheorem{theorem}{Theorem}[section]
\newtheorem{lemma}[theorem]{Lemma}

\newcommand{\EE}{\mathbb{E}}

\title{A diffusion time-changed stochastic SIS epidemic model: well-posedness, long-time behavior, and numerical approximation}

\author{Xiaotong Li\\
Department of Mathematics\\
Jiangsu Second Normal University\\
Nanjing, 210013, China \\
\texttt{x.t.li@foxmail.com} \\
\And
Huaqian Zhou\\
Department of Mathematics\\
Shanghai Normal University\\
Shanghai, 200234, China \\
\texttt{348112933@qq.com} \\
\And
Ruchun Zuo\thanks{Corresponding author} \\
Department of Mathematics\\
Shanghai Normal University\\
Shanghai, 200234, China \\
\texttt{ruchunzuo@outlook.com} \\
}

\renewcommand{\shorttitle}{A Diffusion Time-Changed Stochastic SIS Epidemic Model}

\hypersetup{
pdftitle={A template for the arxiv style},
pdfsubject={q-bio.NC, q-bio.QM},
pdfauthor={David S.~Hippocampus, Elias D.~Striatum},
pdfkeywords={First keyword, Second keyword, More},
}

\begin{document}
\maketitle

\begin{abstract}
In this paper, we propose and analyze a diffusion time-changed susceptible-infected-susceptible (SIS) epidemic model driven by time-changed Brownian motion. We prove that the proposed model admits a unique global positive solution for any initial value in $(0,N)$. The extinction and persistence of the disease are then investigated. To approximate the diffusion time-changed SIS model, we construct a positivity-preserving logarithmic Euler-Maruyama (LEM) method. Assuming that the time-changed is given by the inverse of a standard $\alpha$-stable subordinator with $\alpha\in(0,1)$, we prove that the numerical solution converges strongly to the exact solution with order $\alpha$. Finally, numerical experiments are provided to confirm the predicted convergence rates and illustrate the positivity-preserving property of the proposed method.
\end{abstract}


\section{Introduction}\label{sec:intro}
Mathematical epidemic models play an important role in understanding the transmission mechanisms of infectious diseases, evaluating control strategies, and predicting long-term dynamics. The classical susceptible-infected-susceptible (SIS) model is one of the most fundamental epidemic models and is suitable for diseases in which recovered individuals do not acquire permanent immunity. In the deterministic framework, the dynamical behavior of the SIS model has been well understood \cite{Hethcote2014}. 

In real epidemic systems, disease transmission is inevitably affected by environmental fluctuations, random contact patterns, and uncertainties in model parameters. Therefore, stochastic epidemic models may provide a more appropriate description of disease dynamics in many situations. A common approach is to introduce stochastic perturbations into model parameters. For example, Gray et al. \cite{Gray2011} proposed a stochastic SIS model by perturbing the disease transmission coefficient and obtained the following stochastic differential equation (SDE)
\begin{equation}\label{eq:model31}
	dI(t) = I(t)(\beta N - \mu - \gamma - \beta I(t))dt + \sigma I(t)(N - I(t))dB(t),
\end{equation}
where $\beta$ represents the disease transmission coefficient, $\mu$ and $\gamma$ denote the natural death rate and recovery rate, respectively, $\sigma$ is the environmental noise intensity, and $B(t)$ is a standard Brownian motion.

For stochastic SIS epidemic models, many theoretical results have been obtained. Gray et al. \cite{Gray2011} proposed a stochastic SIS epidemic model and studied the existence of a unique global positive solution, extinction, persistence and stationary distribution. Later, threshold dynamics and long-time behaviors were further investigated for stochastic SIS models with vaccination \cite{ZhaoJiang2014}, standard incidence \cite{LinJiang2014}, stochastic perturbations \cite{LahrouzSettatiAkharif2017}, nonlinear incidence rates \cite{LiuJiangHayatAlsaedi2019,TengWang2016}, and multiple Brownian motions \cite{CaiCaiMao2019,LiGuo2021}. Xu \cite{Xu2017} also established global threshold dynamics for the stochastic SIS model.

In addition to theoretical investigations, the numerical approximation of stochastic SIS epidemic models has also attracted considerable attention. 
Yang and Huang \cite{YangHuang2021} proposed an Euler-Maruyama (EM) method combined with the Lamperti transformation and established its strong convergence of order one in the \(p\)-th moment sense over finite time intervals. Chen, Gan and Wang \cite{Chen2021} proposed an explicit numerical scheme and proved its first-order strong convergence. Several other effective numerical schemes have also been developed for stochastic SIS epidemic models. Yang et al. \cite{YangPanLiuMu2022} studied split-step \(\theta\) methods with truncated Wiener process. Yang et al. \cite{YangLiYangZhang2023} analyzed a linearly backward EM method with truncated Wiener process. Liu, Wang and Dai \cite{LiuWangDai2024} proposed a Milstein type method based on a  logarithmic transformation. Kiouvrekis and Stamatiou \cite{KiouvrekisStamatiou2025} proposed a semi-discrete method for the stochastic SIS epidemic model. More recently, Yang and Huang \cite{YangHuang2024} developed a logarithmic truncated Euler-Maruyama (LTEM) method by incorporating a truncation technique, and showed that the numerical solution can preserve positivity and reproduce the extinction behavior of the exact solution over an infinite time interval.

Although the above studies have greatly advanced the theory and numerical analysis of stochastic SIS epidemic models, most of these works are concerned with Brownian perturbations evolving in the physical time \(t\). In many practical situations, disease transmission may involve random waiting times, temporal heterogeneity and memory effects. Time-changed stochastic processes provide a useful framework for describing such random temporal effects. In particular, when \(E_t\) is the inverse of a subordinator, the time-changed Brownian motion \(B(E_t)\) can be interpreted as Brownian fluctuations evolving under a random operational time. This motivates the study of SIS epidemic models driven by time-changed Brownian motion. Time-changed SDEs have been widely studied in recent years. Kobayashi \cite{Kobayashi2011} developed a stochastic calculus for time-changed semimartingales and derived a corresponding time-changed It\^{o} formula, which provides an important theoretical foundation for time-changed SDEs. Magdziarz \cite{Magdziarz2009} studied stochastic representations of subdiffusion processes with time-dependent drift. Jum and Kobayashi \cite{Jum2016} proposed a strong and weak approximation scheme for SDEs driven by time-changed Brownian motion. Jin and Kobayashi \cite{Jin2019} investigated strong approximation for time-changed SDEs with time-space-dependent coefficients. More related results on time-changed SDEs can be found in \cite{DengLiu2020,Jin2021,LiLiuLiaoXing2023,LiuWangZuo2025,Long2024,WuLiXuPeng2024,WuZuo2026,Zuo2026}.

However, there seem to be few results on diffusion time-changed SIS models driven by time-changed Brownian motion. In particular, the existence of a global positive solution, extinction and persistence, as well as positivity-preserving numerical approximations, have not been fully investigated.

In this paper, we introduce a diffusion time-changed SIS model driven by time-changed Brownian motion. More precisely, we consider
\begin{equation}\label{eq:model33}
	\mathrm{d}I(t) = I(t)\bigl[ \beta N - \mu - \gamma - \beta I(t) \bigr] \mathrm{d}t + \sigma I(t)(N - I(t)) \mathrm{d}B(E_t),
\end{equation}
where \(I(0) = I_0 \in (0,N)\), \(E_t\) is the inverse of a subordinator and is assumed to be independent of the standard Brownian motion \(B(t)\). 

An important feature of model \eqref{eq:model33} is that the drift term evolves in the physical time t, whereas the stochastic perturbation is driven by the time-changed Brownian motion $B(E_t)$. This asymmetric time structure distinguishes model \eqref{eq:model33} from fully time-changed SDEs and prevents the standard duality principle from being directly applied. It therefore brings additional difficulties to both the theoretical analysis and the construction of numerical approximations.

The main contributions of this work are as follows. First, we prove that the SIS model \eqref{eq:model33} admits a unique global positive solution for any initial value in \((0,N)\), which ensures the biological feasibility of the model. Second, we study the long-time dynamical behavior of the solution and establish criteria for extinction and persistence of the disease. Third, we construct a logarithmic Euler-Maruyama (LEM) scheme for the time-changed SIS model and establish its strong convergence of order $\alpha$ when $E_t$ is the inverse of a standard $\alpha$-stable subordinator with $\alpha\in(0,1)$.

The rest of the paper is organized as follows. In Section \ref{sec:preliminary}, the mathematical preparations are presented and some useful lemmas are revisited. The main theoretical results are stated and proved in Section \ref{sec:results}. Section \ref{sec:nummethod} focuses on the construction and analysis of a positivity-preserving numerical method, and establishes its strong convergence. The numerical examples are provided in Section \ref{sec:numerical_simulation}. Conclusions and discussions are presented in Section \ref{sec:conclusion}.

\section{Mathematical preliminaries}\label{sec:preliminary}
In this section, we give some basic settings, notation, and auxiliary results for the rest of the paper.

Throughout this paper, we work on a complete probability space $(\Omega, \mathcal{F},\{\mathcal{F}_t\}_{t\geq 0}, \mathbb{P})$ satisfying the usual conditions, and we let $B(t)$ be a scalar Brownian motion defined on this space. Let $D_t$ be a one-dimensional strictly increasing L\'evy process in $(\Omega,\mathcal{F},\{\mathcal{F}_t\}_{t\geq 0}, \mathbb{P})$ with Laplace transform
\begin{equation*}
	\EE [e^{-rD_t}] = e^{-t \psi(r)},\quad r>0,\,t\geq 0,
\end{equation*}
where the Laplace exponent $\psi:(0,\infty)\rightarrow(0,\infty)$ is a Bernstein function with $\psi(0+):=\lim_{r\downarrow0}\psi(r)=0$. Set $D_0 = 0$, almost surely. Define the inverse subordinator $E_t$ by the inverse of $D_t$ in the way that
\begin{equation*}
	E_t := \inf\{ u\geq0\,;\,D_u > t \}, \quad t \geq 0.
\end{equation*}
It is clear that $E_t$ is continuous and non-decreasing almost surely. Then, $B(E_t)$ is called the time-changed Brownian motion, which is regarded as a sub-diffusion process. If $D$ is a standard $\alpha$-stable subordinator with $\alpha\in(0,1)$, then $\psi(r)=r^{\alpha}$ and $E$ is called the inverse $\alpha$-stable subordinated process. 

In the rest of this paper, we assume that $B(t)$ and $D_t$ are independent. Let $\mathbb{E}$ denote the expectation corresponding to $\mathbb{P}$. Moreover, \(\mathbb{E}_{B}\) and \(\mathbb{E}_{D}\) denote the expectations with respect to \(B(t)\) and \(D_t\), respectively. For every integrable random variable $X$,  independence of \(B\) and \(D\) yields $\EE[X]=\EE_B[\EE_D[X]]=\EE_D[\EE_B[X]]$. Throughout the subsequent analysis, $\EE_B$ denotes expectation with respect to the Brownian motion for a fixed realization of $D$. In addition, for any $a,b\in\mathbb{R}$, we use the notation $a\vee b:=\max\{a,b\}$ and $a\wedge b:=\min\{a,b\}$. 


	

The following two lemmas are taken from \cite{Jin2019,Nane2016} and will play an important role in the subsequent analysis.
\begin{lemma}\label{lemma:211}
	For an $\alpha$-stable inverse subordinator $E_t$, we have
	\[
	\lim_{t\to\infty} \frac{E_t}{t} = 0 \quad \text{a.s.}
	\]
\end{lemma}

\begin{lemma}\label{lemma:29} 
	Let \( E \) be the inverse of a subordinator \( D \) whose Laplace exponent \( \psi \) is regularly varying at \( \infty \) with index \( \rho \in [0, 1) \). If \( \rho = 0 \), assume further that \( \nu(0, \infty) = \infty \). Fix \( \lambda > 0 \), \( t > 0 \) and \( r > 0 \).
	\begin{enumerate}[label=(\arabic*)]
		\item If $r < 1/(1 - \rho)$, then $ \EE[e^{\lambda E_t^r}] < \infty $.
		\item If $ r > 1/(1 - \rho)$, then $ \EE[e^{\lambda E_t^r}] = \infty $.
	\end{enumerate}
\end{lemma}

\section{Model properties}\label{sec:results}
In this section, we investigate the fundamental properties and long-term dynamics of the diffusion time-changed SIS model \eqref{eq:model33}. We first establish the existence and uniqueness of a global solution and show that the solution remains in the biologically meaningful interval $(0,N)$ almost surely. We then derive an extinction criterion in the subcritical case and further characterize the long-term behavior of the solution in the supercritical case.

\begin{theorem}
	For any initial value \(I(0) = I_0 \in (0,N)\), the diffusion time-changed SIS model admits a unique global positive solution \(I(t) \in (0,N)\) for all \(t \ge 0\) almost surely, i.e.,
	\[
	\mathbb{P}\{I(t) \in (0,N), \ \forall t \ge 0\} = 1.
	\]
\end{theorem}

\begin{proof}
	Viewing the diffusion time-changed SIS model \eqref{eq:model33} as an SDE on \(\mathbb{R}\), its coefficients are locally Lipschitz continuous. Hence, for any given initial value \(I_0 \in (0,N)\), there exists a unique maximal local solution \(I(t)\) on \(t \in [0, \tau_e)\), where \(\tau_e\) is the explosion time (see e.g., \cite{Protter2004}). To show that the solution is global, we need to prove \(\tau_e = \infty\) almost surely.
	
	Choose integer \(k_0 \geq 1\) sufficiently large such that \(1/k_0 < I_0 < N - 1/k_0\). For each integer \(k \ge k_0\), define the stopping time
	\[
	\tau_k = \inf\{t \in [0, \tau_e) : I(t) \notin (1/k, N-1/k)\}.
	\]
	Clearly, \(\tau_k\) is increasing as \(k \to \infty\). Set \(\tau_\infty = \lim_{k\to\infty} \tau_k\), then \(\tau_\infty \le \tau_e\) almost surely. If we can show \(\tau_\infty = \infty\) almost surely, then \(\tau_e = \infty\) almost surely and \(I(t) \in (0,N)\) for all \(t \ge 0\).
	
	Suppose, to the contrary, that \(\tau_\infty < \infty\) with positive probability. Then there exist constants \(T > 0\) and \(\varepsilon \in (0,1)\) such that
	\[
	\mathbb{P}\{\tau_\infty \le T\} > \varepsilon.
	\]
	Consequently, there exists an integer \(k_1 \ge k_0\) such that
	\begin{equation}
		\mathbb{P}\{\tau_k \le T\} \ge \varepsilon \quad \forall~ k \ge k_1. \label{eq:model3.4}
	\end{equation}
	Define a Lyapunov function \(V: (0,N) \to \mathbb{R}_+\) by
	\[
	V(x) = \frac{1}{x} + \frac{1}{N-x}.
	\]
	
	Applying the time-changed It\^o formula to \(V(I(t\wedge \tau_k))\), and then taking expectations, where the stochastic integral term vanishes due to the independence of \(B\) and \(E\), we obtain, for any \(t\in[0,T]\) and \(k\ge k_1\),
	\begin{equation}\label{eq:E_BV}
		\mathbb{E}_B\left[V(I(t\wedge \tau_k))\right] = V(I_0) + \mathbb{E}_B\left[\int_{0}^{t\wedge\tau_k} L_1 V(I(s)) \mathrm{d}t\right] + \mathbb{E}_B\left[\int_{0}^{t\wedge\tau_k} L_2 V(I(s)) \mathrm{d}E_t\right],
	\end{equation}
	where $L_1 V(x) = V'(x) \left[x(\beta N - \mu - \gamma - \beta x)\right]$ and $L_2 V(x) = \tfrac12 V''(x) \sigma^2 x^2 (N-x)^2$.
	
	It is easy to show that
	\begin{align*}
		L_1 V(x) &=\left(-\frac{1}{x^2}+\frac{1}{\left(N-x\right)^2}\right)x\left(\beta(N-x)-(\mu+\gamma)\right)\\
		&=\frac{\mu+\gamma}{x}-\frac{\beta(N-x)}{x}+\frac{\beta x}{N-x}-\frac{(\mu+\gamma)x}{(N-x)^2}\\
		&\leq \frac{\mu+\gamma}{x}+\frac{\beta N}{N-x}\\
		&= C_1 V(x)
	\end{align*}
	and
	\begin{align*}
		L_2V(x)&=\frac{1}{2}\left(\frac{2}{x^3}+\frac{2}{(N-x)^3}\sigma^2x^2(N-x)^2\right)\\
		&\leq \frac{\sigma^2N^2}{x}+\frac{\sigma^2N^2}{N-x}\\
		&=C_2 V(x),
	\end{align*}
	where $C_1=(\mu+\gamma)\vee\beta N,~ C_2=\sigma^2N^2$.
	
	Substituting this into \eqref{eq:E_BV}, we have
	\[
	\mathbb{E}_B \left[V(I(t \wedge \tau_k))\right] \le V(I_0) + C_1 \int_0^t \mathbb{E}_B[V(I(s \wedge \tau_k))] \mathrm{d}s + C_2 \int_0^t \mathbb{E}_B[V(I(s \wedge \tau_k))] \mathrm{d}E_s.
	\]
	Applying the time-changed Gronwall inequality yields
	\begin{equation*}
		\mathbb{E}_B \left[V(I(t \wedge \tau_k))\right] \le V(I_0) e^{C_1 T + C_2 E_T}.
	\end{equation*}
	Taking expectations with respect to the time-change process \(D\) on both sides, we further obtain
	\begin{align}\label{eq:EV}
		\mathbb{E}\left[V(I(t\wedge\tau_k))\right]
		&=\mathbb{E}_{D}\left[\mathbb{E}_{B}V(I(t\wedge\tau_k))\right]   \nonumber\\
		&\leq V(I_0)e^{C_1T}\mathbb{E}_{D}\left[e^{C_2E_T}\right].
	\end{align}
	
	Now, let \(\Omega_k = \{\tau_k \le T\}\) for $k\geq k_1$. It follows from \eqref{eq:model3.4} that \(\mathbb{P}(\Omega_k) \ge \varepsilon\). For every \(\omega \in \Omega_k\), \(I(\tau_k, \omega)\) equals either \(1/k\) or \(N-1/k\), so
	\[
	V(I(\tau_k, \omega)) \ge k.
	\]
	From \eqref{eq:EV}, we obtain
	\[
	\mathbb{E}[V(I(\tau_k \wedge T))] \ge \mathbb{E}[I_{\Omega_k} V(I(\tau_k))] \ge k \mathbb{P}(\Omega_k) \ge \varepsilon k.
	\]
	Combining these and Lemma \ref{lemma:29}, we have
	\[
	\varepsilon k \le V(I_0)e^{C_1T}\mathbb{E}_{D}\left[e^{C_2E_T}\right] < \infty \quad \forall~ k \ge k_1,
	\]
	which leads to a contradiction as \(k \to \infty\). Hence \(\tau_\infty = \infty\) almost surely, completing the proof.
\end{proof}
With the global existence established, we now investigate conditions for disease extinction.
\begin{theorem}
	Let \(\mathcal{R}_0 = \beta N / (\mu + \gamma)\). If \(\mathcal{R}_0  < 1\), then for any initial value \(I(0) = I_0 \in (0,N)\), the solution of diffusion time-changed SIS model \eqref{eq:model33} satisfies
	\[
	\limsup_{t\to\infty} \frac{1}{t} \log I(t) < 0 \quad \text{a.s.}
	\]
	In other words, the disease dies out exponentially with probability one.
\end{theorem}

\begin{proof}
	Applying the time-changed It\^o formula to \(\log I(t)\), we obtain
	\begin{equation}
		\begin{aligned}
			\log I(t) &= \log I_0 + \int_0^t (\beta N - \mu - \gamma - \beta I(s)) \mathrm{d}s \\
			&\quad - \frac12 \sigma^2 \int_0^t (N - I(s))^2 \mathrm{d}E_s + \int_0^t \sigma (N - I(s)) \mathrm{d}B(E_s).
		\end{aligned}
		\label{eq:model3.5}
	\end{equation}
	Since \(I(s) \in (0,N)\), we have the estimates
	\[
	\int_0^t (\beta N - \mu - \gamma - \beta I(s)) \mathrm{d}s \le (\beta N - \mu - \gamma) t
	\]
	and
	\[
	-\frac12 \sigma^2 \int_0^t (N - I(s))^2 \mathrm{d}E_s \le -\frac12 \sigma^2 N^2 E_t + \sigma^2 N \int_0^t I(s) \mathrm{d}E_s.
	\]
	Substituting these into (\ref{eq:model3.5}) gives
	\begin{equation}
		\log I(t) \le \log I_0 + (\beta N - \mu - \gamma) t - \frac12 \sigma^2 N^2 E_t + \sigma^2 N \int_0^t I(s) \mathrm{d}E_s + M_t, \label{eq:model3.6}
	\end{equation}
	where \(M_t = \int_0^t \sigma (N - I(s)) \mathrm{d}B(E_s)\).
	
	Dividing both sides of (\ref{eq:model3.6}) by \(t\) and taking the limit superior as \(t \to \infty\), we obtain
	\begin{align*}
		\limsup_{t\to\infty} \frac{1}{t} \log I(t) &\le \beta N - \mu - \gamma - \frac12 \sigma^2 N^2 \limsup_{t\to\infty} \frac{E_t}{t} \\
		&\quad + \sigma^2 N \limsup_{t\to\infty} \frac{1}{t} \int_0^t I(s) dE_s + \limsup_{t\to\infty} \frac{M_t}{t}\quad \text{a.s.}
	\end{align*}
	By Lemma \ref{lemma:211}, we have
	\[
	\lim_{t\to\infty}\frac{E_t}{t}=0
	\quad \text{a.s.}
	\]
	Moreover, since \(I(s)\in(0, N)\), it follows that
	\[
	\limsup_{t\to\infty} \frac{1}{t} \int_0^t I(s) dE_s \le N \limsup_{t\to\infty} \frac{E_t}{t} = 0 \quad \text{a.s.}
	\]
	
	
	
	By the strong law of large numbers for continuous local martingales, we obtain
	\[
	\frac{M_t}{t}\to0
	\quad \text{a.s.}
	\]
	Consequently,
	\[
	\limsup_{t\to\infty}\frac1t\log I(t)
	\leq \beta N-\mu-\gamma \quad \text{a.s.}
	\]
	Since \(\mathcal{R}_0 <1\), namely \(\beta N-\mu-\gamma<0\), the desired assertion follows.
\end{proof}

Finally, we examine the persistence scenario when the reproduction number exceeds unity.
\begin{theorem}\label{Persistence}
	If \(\mathcal{R}_0  > 1\), then for any initial value \(I(0) = I_0 \in (0,N)\), the solution of diffusion time-changed SIS model \eqref{eq:model33} satisfies
	\[
	\limsup_{t\to\infty} I(t) \ge \xi \quad \text{and}\quad \liminf_{t\to\infty} I(t) \le \xi \quad \text{a.s.},
	\]
	where \(\xi = N(1 - 1/\mathcal{R}_0 )\) is the endemic equilibrium of the deterministic model.
\end{theorem}

\begin{proof}
	From \eqref{eq:model3.5}, we write
	\begin{equation}\label{eq:logIt}
		\log I(t) = \log I_0 + A_t + C_t + M_t,
	\end{equation}
	with
	\begin{align*}
		A_t &= \int_0^t (\beta N - \mu - \gamma - \beta I(s)) \mathrm{d}s, \\
		C_t &= -\frac12 \sigma^2 \int_0^t (N - I(s))^2 \mathrm{d}E_s, \\
		M_t &= \int_0^t \sigma (N - I(s)) \mathrm{d}B(E_s).
	\end{align*}
	
	By the strong law for martingales, \(\lim_{t\to\infty} M_t/t = 0\) a.s. For \(C_t\), since \(N - I(s) \le N\), we have
	\begin{equation*}
		-\frac12 \sigma^2 N^2 E_t\leq C_t\leq 0.
	\end{equation*}
	Again using Lemma \ref{lemma:211}, we get \(\lim_{t\to\infty} C_t/t = 0\) a.s.
	
	Now, define \(f(x) = \beta N - \mu - \gamma - \beta x\). Note that \(f(x)\) is strictly decreasing in \(x\) and \(f(\xi) = 0\). We first prove \(\limsup_{t\to\infty} I(t) \ge \xi\) a.s. Suppose, to the contrary, that there exists \(\varepsilon > 0\) such that
	\[
	\mathbb{P}\left(\limsup_{t\to\infty} I(t) \le \xi - 2\varepsilon\right) > 0.
	\]
	Then, for every \(\omega\) in this event, there exists a random time \(T_{\omega} > 0\) such that
	\begin{equation}\label{eq:Itomega}
		I(t,\omega)\leq \xi-\varepsilon, \quad t\geq T_{\omega}.
	\end{equation}
	Since $f(x)=\beta(\xi-x)$, we have
	\[
	f(\xi-\varepsilon)=\beta\varepsilon=:\eta>0.
	\]
	Consequently, for all \(t\geq T_{\omega}\), 
	\[
	f(I(t,\omega))\geq f(\xi-\varepsilon)=\eta.
	\]
	Hence,
	\[
	A_t \ge \int_0^{T_{\omega}} f(I(s)) \mathrm{d}s + \eta (t - T_{\omega}).
	\]
	Dividing both sides of \eqref{eq:logIt} by \(t\) and letting \(t \to \infty\), we obtain
	\[
	\liminf_{t\to\infty} \frac{\log I(t)}{t} \ge \eta > 0,
	\]
	which implies \(I(t) \to \infty\), contradicting \eqref{eq:Itomega}. Hence, 
	\[
	\limsup_{t\to\infty} I(t) \ge \xi \quad \text{a.s.}
	\]
	holds. The proof of \(\liminf_{t\to\infty} I(t) \le \xi\) a.s. is similar. Assume there exists \(\delta > 0\) such that
	\[
	\mathbb{P}\left(\liminf_{t\to\infty} I(t) \ge \xi + 2\delta\right) > 0.
	\]
	Then there exists \(T'_{\omega} > 0\) such that 
	\begin{equation}\label{eq:Itomega2}
		I(t,\omega) \ge \xi + \delta,\quad  t \ge T'_{\omega}.
	\end{equation}
	Note that \(f(\xi + \delta) =-\beta\delta=: -\zeta < 0\). Hence for \(t \ge T'_{\omega}\), we have
	\[
	A_t \le \int_0^{T'_{\omega}} f(I(s)) \mathrm{d}s - \zeta (t - T'_{\omega}).
	\]
	
	Again, dividing by \(t\) and taking the limit superior yields
	\[
	\limsup_{t\to\infty} \frac{\log I(t)}{t} \le -\zeta < 0,
	\]
	which implies \(I(t) \to 0\), contradicting \eqref{eq:Itomega2}. Therefore, \(\liminf_{t\to\infty} I(t) \le \xi\) a.s. The
	proof is therefore complete.
\end{proof}
\section{Strong convergence}\label{sec:nummethod}
In this section, we establish strong convergence of the LEM approximation for the diffusion time-changed SIS model \eqref{eq:model33}. Since the drift term evolves in physical time while only the diffusion term is driven by the time-changed Brownian motion, the standard duality principle \cite{Jum2016} for fully time-changed SDEs is not directly applicable to model \eqref{eq:model33}. Therefore, we construct and analyze the numerical approximation directly on a uniform physical-time grid.

Throughout this section, $p\geq2$ denotes a fixed integer whenever a $p$-th moment estimate is stated.
Unless otherwise specified, $C>0$ denotes a generic positive constant that may depend on the parameters $p$, $T$, $\alpha$, $\beta$, $\mu$, $\gamma$, $\sigma$, and $N$, but is always independent of the discretization step sizes $\Delta$ and $\delta$.
Its value may vary from line to line.

We introduce the transformed state variable
\begin{equation*}
	y(t) := \log\frac{I(t)}{N-I(t)}, \qquad q := \sigma N,
\end{equation*}
which is a logarithmic Lamperti type transformation. Since the quadratic variation of the time-changed Brownian motion $B(E_t)$ is $E_t$, applying the time-changed It\^{o} formula yields
\begin{equation}\label{eq:transformed2}
	\mathrm{d}y(t)=F(y(t))\mathrm{d}t+G(y(t))\mathrm{d}E_t+q\mathrm{d}B(E_t),
\end{equation}
where the drift and diffusion coefficients are defined by
\begin{equation}\label{eq:def-FG}
	F(x)=\beta N-\mu-\gamma-(\mu+\gamma)e^x,\qquad G(x)=\frac{q^2}{2}-\frac{q^2}{1+e^x}.
\end{equation}
Direct differentiation shows that the derivatives of $F$ and $G$ are given by
\begin{equation*}
	F'(x)=F''(x)=-(\mu+\gamma)e^x,
\end{equation*}
and
\begin{equation*}
	G'(x)=q^2\frac{e^x}{(1+e^x)^2},\qquad G''(x)=q^2\frac{e^x(1-e^x)}{(1+e^x)^3}.
\end{equation*}
Since $\sup_{u>0}\frac{u}{(1+u)^2} = \frac{1}{4}$ and $\sup_{u>0}\frac{u|1-u|}{(1+u)^3} \leq \frac{1}{4}$, the function $G$ and its derivatives are uniformly bounded on $\mathbb R$:
\begin{equation*}
	|G(x)|\leq\frac{q^2}{2},\qquad0\leq G'(x)\leq\frac{q^2}{4},\qquad|G''(x)|\leq\frac{q^2}{4}.
\end{equation*}
Consequently, for some generic constant $C > 0$, the terms involving $F$ and its derivatives satisfy the exponential growth bound
\begin{equation}\label{eq:F-growth-bounds}
	|F'(x)|+|F'(x)F(x)|+\left|F'(x)G(x)+\frac{q^2}{2}F''(x)\right|+|G'(x)F(x)|\leq C(1+e^{2x}),
\end{equation}
while the compound terms involving purely $G$ and its derivatives remain uniformly bounded on $\mathbb R$:
\begin{equation}\label{eq:G-growth-bounds}
	|G(x)|+|G'(x)|+|G''(x)|+\left|G'(x)G(x)+\frac{q^2}{2}G''(x)\right|\leq C,
\end{equation}
for all $x \in \mathbb R$.

Let $\Delta \in (0,1]$ be a uniform step size, and set $t_k = k\Delta$ for $k = 0, 1, \dots, n_\Delta$ with $n_\Delta = \lfloor T/\Delta \rfloor$.
We denote the increments of the time change and Brownian motion over each subinterval by
\begin{equation*}
	\Delta E_k=E_{t_{k+1}}-E_{t_k},\qquad\Delta B_k=B(E_{t_{k+1}})-B(E_{t_k}).
\end{equation*}
The explicit LEM scheme for \eqref{eq:transformed2} is then defined as
\begin{equation}\label{eq:numerical-scheme2}
	X_{k+1}=X_k+F(X_k)\Delta+G(X_k)\Delta E_k+q\Delta B_k,\qquad X_0=y(0),
\end{equation}
and the numerical approximation in the original variable is recovered via
\begin{equation}\label{eq:inverse-diff}
	I_k=\frac{Ne^{X_k}}{1+e^{X_k}}.
\end{equation}
This formulation guarantees $0<I_k<N$ for all $k\geq0$.

\begin{lemma}\label{lem:exp-moment-bounds}
	Let $y(t)$ be the exact solution of \eqref{eq:transformed2} and let $X_k$ be generated by \eqref{eq:numerical-scheme2}.
	For every $p>0$ and $T>0$, there exists a positive constant $C$, independent of $\Delta$, such that
	\begin{align*}
		\sup_{0\leq t\leq T}\mathbb E_B\left[e^{p|y(t)|}\right] &\leq C\exp\{C(T+E_T)\},\\
		\max_{0\leq k\leq n_\Delta}\mathbb E_B\left[e^{pX_k}\right] &\leq C\exp\{C(T+E_T)\}.
	\end{align*}
\end{lemma}

\begin{proof}
	We first establish the exponential moment bound for the exact solution $y(t)$.
	Let $I(t)$ solve \eqref{eq:model33} and define the stopping time
	\begin{equation*}
		\vartheta_m = \inf\{t\geq0 : I(t)\notin(1/m, N-1/m)\}.
	\end{equation*}
	For $V_{\mathrm{I}}(x)=x^{-p}$,
	\begin{equation*}
		V_{\mathrm{I}}'(x)=-px^{-p-1},\qquad V_{\mathrm{I}}''(x)=p(p+1)x^{-p-2}.
	\end{equation*}
	Using the time-changed It\^{o} formula, we have
	\begin{align}\label{eq:diff-ito-I}
		I(t\wedge\vartheta_m)^{-p}
		&= I_0^{-p} - p\int_0^{t\wedge\vartheta_m} I(s)^{-p}[\beta N-\mu-\gamma-\beta I(s)]\mathrm{d}s \notag\\
		&\quad + \frac{p(p+1)}{2}\sigma^2 \int_0^{t\wedge\vartheta_m} I(s)^{-p}(N-I(s))^2\mathrm{d}E_s  \notag\\
		&\quad - p\sigma\int_0^{t\wedge\vartheta_m} I(s)^{-p}(N-I(s))\mathrm{d}B(E_s).
	\end{align}
	Since $0 < I < N$, we have
	\begin{equation*}
		-p[\beta N-\mu-\gamma-\beta I]=p[\mu+\gamma-\beta(N-I)]\leq p(\mu+\gamma)=:c_1.
	\end{equation*}
	and
	\begin{equation*}
		\frac{p(p+1)}{2}\sigma^2(N-I)^2\leq\frac{p(p+1)}{2}\sigma^2N^2=:c_2.
	\end{equation*}
	By the definition of the stopping time $\vartheta_m$, the stopped stochastic integral in \eqref{eq:diff-ito-I} is a martingale with zero $\mathbb E_B$-expectation.
	Taking $\mathbb E_B$ on both sides of \eqref{eq:diff-ito-I} and using the preceding estimates, we obtain
	\begin{align*}
		\mathbb E_B\left[I(t\wedge\vartheta_m)^{-p}\right] \leq I_0^{-p} + c_1\int_0^t \mathbb E_B\left[I(s\wedge\vartheta_m)^{-p}\right]\mathrm{d}s + c_2\int_0^t \mathbb E_B\left[I(s\wedge\vartheta_m)^{-p}\right]\mathrm{d}E_s.
	\end{align*}
	Then, by the time-changed Gronwall inequality, we have
	\begin{equation*}
		\mathbb E_B\left[I(t\wedge\vartheta_m)^{-p}\right] \leq I_0^{-p}\exp\{c_1t+c_2E_t\}.
	\end{equation*}
	Letting $m\to\infty$ and applying Fatou's lemma, we obtain, for every $t\in[0,T]$,
	\begin{equation*}
		\mathbb E_B\left[I(t)^{-p}\right] \leq I_0^{-p}\exp\{c_1t+c_2E_t\}.
	\end{equation*}
	Since $t\leq T$ and $E_t\leq E_T$, taking the supremum over $t\in[0,T]$ gives
	\begin{equation*}
		\sup_{0\leq t\leq T}\mathbb E_B\left[I(t)^{-p}\right] \leq I_0^{-p}\exp\{c_1T+c_2E_T\}.
	\end{equation*}
	
	To obtain the corresponding bound for $N-I(t)$, we set $J(t)=N-I(t)$ and $V_{\mathrm{J}}(I)=(N-I)^{-p}=J^{-p}$.
	From \eqref{eq:model33},
	\begin{equation*}
		\mathrm{d}J(t)=-I(t)[\beta J(t)-(\mu+\gamma)]\mathrm{d}t-\sigma I(t)J(t)\mathrm{d}B(E_t).
	\end{equation*}
	Using the time-changed It\^{o} formula, we have
	\begin{align}\label{eq:diff-ito-J}
		J(t\wedge\vartheta_m)^{-p}
		&= J(0)^{-p} + p\int_0^{t\wedge\vartheta_m} I(s)[\beta J(s)-(\mu+\gamma)]J(s)^{-p-1}\mathrm{d}s \notag\\
		&\quad + \frac{p(p+1)}{2}\sigma^2 \int_0^{t\wedge\vartheta_m} I(s)^2J(s)^{-p}\mathrm{d}E_s + p\sigma\int_0^{t\wedge\vartheta_m} I(s)J(s)^{-p}\mathrm{d}B(E_s).
	\end{align}
	Since $0 < I < N$, direct calculation shows that
	\begin{equation*}
		pI[\beta J-(\mu+\gamma)]J^{-p-1}\leq p\beta I J^{-p}\leq c_3J^{-p},\qquad\frac{p(p+1)}{2}\sigma^2I^2J^{-p}\leq c_2J^{-p},
	\end{equation*}
	where $c_3 = p\beta N$.
	Similarly, the stopped stochastic integral in \eqref{eq:diff-ito-J} is a martingale with zero $\mathbb E_B$-expectation.
	Taking $\mathbb E_B$ on both sides of \eqref{eq:diff-ito-J} and applying Gronwall's inequality and Fatou's lemma as $m\to\infty$, we obtain
	\begin{equation*}
		\sup_{0\leq t\leq T}\mathbb E_B\left[(N-I(t))^{-p}\right] \leq (N-I_0)^{-p}\exp\{c_3T+c_2E_T\}.
	\end{equation*}
	
	Since $y(t) = \log\frac{I(t)}{N-I(t)}$, we have
	\begin{equation*}
		e^{py(t)}\leq N^p(N-I(t))^{-p},\qquad e^{-py(t)}\leq N^pI(t)^{-p}.
	\end{equation*}
	Using the elementary inequality $e^{p|x|}\leq e^{px}+e^{-px}$ for $x\in\mathbb R$, we deduce that
	\begin{align*}
		\sup_{0\leq t\leq T}\mathbb E_B\left[e^{p|y(t)|}\right] &\leq N^p\left[\sup_{0\leq t\leq T}\mathbb E_B\left[(N-I(t))^{-p}\right] + \sup_{0\leq t\leq T}\mathbb E_B\left[I(t)^{-p}\right]\right] \\
		&\leq C\exp\{C(T+E_T)\}.
	\end{align*}
	
	We next establish the exponential moment bound for the numerical solution $X_k$.
	By \eqref{eq:def-FG}, there exist constants $C_1,C_2>0$ such that
	\begin{equation*}
		F(x)\leq C_1,\qquad G(x)\leq C_2,\qquad x\in\mathbb R.
	\end{equation*}
	Iteration of \eqref{eq:numerical-scheme2} therefore gives
	\begin{equation*}
		X_k\leq X_0+C_1t_k+C_2E_{t_k}+qB(E_{t_k})\leq X_0+C_1T+C_2E_T+qB(E_{t_k}).
	\end{equation*}
	Taking the exponential on both sides of the preceding inequality and then taking $\mathbb E_B$, we obtain from the moment-generating function of Brownian motion that
	\begin{align*}
		\mathbb E_B\left[e^{pX_k}\right]&\leq\exp\{pX_0+pC_1T+pC_2E_T\}\mathbb E_B\left[e^{pqB(E_{t_k})}\right]\\
		&=\exp\left\{pX_0+pC_1T+pC_2E_T+\frac{p^2q^2}{2}E_{t_k}\right\}\\
		&\leq\exp\{pX_0+pC_1T+C_3E_T\}\leq C\exp\{C(T+E_T)\},
	\end{align*}
	where $C_3=pC_2+p^2q^2/2$.
	Taking the maximum over $0\leq k\leq n_\Delta$ completes the proof.
\end{proof}

\begin{lemma}\label{lem:sharp-clock}
	Suppose that $E$ is the inverse of a standard $\alpha$-stable subordinator with index $\alpha\in(0,1)$.
	For every fixed integer $p\geq2$, $c>0$, and $T>0$, there exists a positive constant $C$, independent of $\Delta$, such that
	\begin{equation*}
		\mathbb E\left[e^{cE_T}\sum_{k=0}^{n_\Delta-1}(\Delta E_k)^{p+1}\right]\leq C\Delta^{p\alpha}.
	\end{equation*}
\end{lemma}

\begin{proof}
	Let $\ell=p+1$, let $P$ be a rate-one Poisson process independent of $D$, and define
	\begin{equation*}
		\mathcal N_t=P(E_t).
	\end{equation*}
	Since $P$ and $D$ are independent and $E$ is measurable with respect to $D$, we use $\mathbb E_P$ and $\mathbb E_D$ for the corresponding partial expectations.
	Taking $\mathbb E_P$ first, $\mathcal N$ is a Poisson point process on physical time with intensity measure $\mathrm{d}E_t$.
	After taking $\mathbb E_D$, $\mathcal N$ is a renewal process.
	
	Let $\varsigma$ be an exponential random variable with mean one and independent of $D$.
	The renewal waiting time has the same distribution as $D(\varsigma)$ \cite{Meerschaert2011}.
	Let $f_0$ denote the density of $D(\varsigma)$.
	Its Laplace transform is
	\begin{equation*}
		\widehat f_0(\lambda)=\mathbb E\left[e^{-\lambda D(\varsigma)}\right]=\int_0^\infty e^{-s}e^{-s\lambda^\alpha}\mathrm{d}s=\frac{1}{1+\lambda^\alpha}.
	\end{equation*}
	
	Fix $z=1+c>1$ and write $(m)_{[\ell]}=m(m-1)\cdots(m-\ell+1)$.
	The factorial Campbell formula \cite{breton2014factorial} gives, for $0\leq a<b\leq T$,
	\begin{equation}\label{eq:conditional-campbell}
		\mathbb E_P\left[\bigl(\mathcal N(a,b]\bigr)_{[\ell]}z^{\mathcal N_T-\ell}\right]=(E_b-E_a)^\ell e^{(z-1)E_T}.
	\end{equation}
	
	Define the weighted renewal density
	\begin{equation*}
		r_z(t)=\sum_{m=1}^\infty z^{m-1}f_0^{*m}(t),
	\end{equation*}
	where $f_0^{*m}$ denotes the $m$-fold convolution of $f_0$.
	For every sufficiently large $\lambda>(z-1)^{1/\alpha}$, the convolution property and the geometric-series formula give
	\begin{equation*}
		\widehat r_z(\lambda)=\sum_{m=1}^\infty z^{m-1}\bigl(\widehat f_0(\lambda)\bigr)^m=\frac{\widehat f_0(\lambda)}{1-z\widehat f_0(\lambda)}=\frac{1}{\lambda^\alpha-(z-1)}.
	\end{equation*}
	The inverse Laplace transform therefore yields
	\begin{equation*}
		r_z(t)=t^{\alpha-1}E_{\alpha,\alpha}((z-1)t^\alpha)\leq C_{\alpha,z,T}t^{\alpha-1},\qquad0<t\leq T,
	\end{equation*}
	where $E_{\alpha,\alpha}$ is the two-parameter Mittag--Leffler function.
	
	The inverse-stable moment formula gives
	\begin{equation*}
		\mathbb E_D\left[E_t^m\right]=\frac{m!t^{\alpha m}}{\Gamma(1+\alpha m)},\qquad m=0,1,2,\ldots.
	\end{equation*}
	Consequently, the probability generating function of $\mathcal N_t$ is
	\begin{align*}
		H_z(t)&:=\mathbb E\left[z^{\mathcal N_t}\right]=\mathbb E_D\left[e^{(z-1)E_t}\right]=E_{\alpha,1}((z-1)t^\alpha).
	\end{align*}
	The continuity of the Mittag--Leffler function implies that $H_z$ is bounded on $[0,T]$.
	
	We next order the $\ell$ selected renewal epochs as
	\begin{equation*}
		a<s_1<s_2<\cdots<s_\ell\leq b.
	\end{equation*}
	The renewal property and the $\ell!$ possible permutations give
	\begin{align*}
		\mathbb E &\left[
		\bigl(\mathcal N(a,b]\bigr)_{[\ell]}z^{\mathcal N_T-\ell}
		\right]\\
		&=\ell!\int_a^b
		\int_{s_1}^b
		\cdots
		\int_{s_{\ell-1}}^b
		r_z(s_1)
		\prod_{j=2}^{\ell}r_z(s_j-s_{j-1})
		H_z(T-s_\ell)
		\mathrm{d}s_\ell\cdots\mathrm{d}s_1.
	\end{align*}
	Taking $\mathbb E_D$ on both sides of \eqref{eq:conditional-campbell} and applying the bounds for $r_z$ and $H_z$, we obtain
	\begin{align*}
		\mathbb E& \left[e^{cE_T}(E_b-E_a)^{p+1}\right]\\
		&\leq
		C\int_a^b s_1^{\alpha-1}
		\int_0^{b-s_1}
		\int_0^{b-s_1-u_2}\cdots
		\int_0^{b-s_1-\sum_{j=2}^{p}u_j}
		\prod_{j=2}^{p+1}u_j^{\alpha-1}
		\mathrm{d}u_{p+1}\cdots\mathrm{d}u_2\mathrm{d}s_1.
	\end{align*}
	The Dirichlet integral satisfies
	\begin{equation*}
		\int_0^{b-s_1}
		\int_0^{b-s_1-u_2}\cdots
		\int_0^{b-s_1-\sum_{j=2}^{p}u_j}
		\prod_{j=2}^{p+1}u_j^{\alpha-1}
		\mathrm{d}u_{p+1}\cdots\mathrm{d}u_2
		=\frac{\Gamma(\alpha)^p}{\Gamma(p\alpha+1)}(b-s_1)^{p\alpha}.
	\end{equation*}
	It follows that
	\begin{equation}\label{eq:weighted-clock-increment}
		\mathbb E\left[e^{cE_T}(E_b-E_a)^{p+1}\right]\leq C(b-a)^{p\alpha}\int_a^b s^{\alpha-1}\mathrm{d}s.
	\end{equation}
	Applying \eqref{eq:weighted-clock-increment} with $a=t_k$ and $b=t_{k+1}$ gives
	\begin{equation*}
		\mathbb E\left[e^{cE_T}(\Delta E_k)^{p+1}\right]\leq C\Delta^{p\alpha}\int_{t_k}^{t_{k+1}}s^{\alpha-1}\mathrm{d}s.
	\end{equation*}
	Therefore,
	\begin{align*}
		\mathbb E\left[e^{cE_T}\sum_{k=0}^{n_\Delta-1}(\Delta E_k)^{p+1}\right]&\leq C\Delta^{p\alpha}\int_0^T s^{\alpha-1}\mathrm{d}s\leq C\Delta^{p\alpha}.
	\end{align*}
	This completes the proof.
\end{proof}

\begin{theorem}\label{thm:strong-conv-diffusion}
	Suppose that $E$ is the inverse of a standard $\alpha$-stable subordinator with index $\alpha\in(0,1)$ and is independent of $B$.
	Let $p\geq2$ be a fixed integer, let $y(t)$ be the exact solution of \eqref{eq:transformed2}, and let $X_k$ be generated by \eqref{eq:numerical-scheme2}.
	Then there exists a positive constant $C$, independent of $\Delta$, such that
	\begin{equation*}
		\max_{0\leq k\leq n_\Delta}\mathbb E\left[|X_k-y(t_k)|^p\right]\leq C\Delta^{p\alpha}.
	\end{equation*}
	Moreover, the approximation $I_k$ defined by \eqref{eq:inverse-diff} satisfies
	\begin{equation*}
		\max_{0\leq k\leq n_\Delta}\mathbb E\left[|I_k-I(t_k)|^p\right]\leq C\Delta^{p\alpha}.
	\end{equation*}
\end{theorem}

\begin{proof}
	Let
	\begin{equation*}
		y_k=y(t_k),\qquad h_k=\Delta+\Delta E_k.
	\end{equation*}
	The constant $C>0$ used below may change from line to line and may depend on $p$, $T$, and the model parameters, but it is independent of $\Delta$.
	
	The integral form of \eqref{eq:transformed2} gives
	\begin{equation}\label{eq:diff-exact-recursion}
		y_{k+1}=y_k+F(y_k)\Delta+G(y_k)\Delta E_k+q\Delta B_k-\mathcal R_{k+1},
	\end{equation}
	where
	\begin{equation*}
		\mathcal R_{k+1}=\int_{t_k}^{t_{k+1}}[F(y_k)-F(y(s))]\mathrm{d}s+\int_{t_k}^{t_{k+1}}[G(y_k)-G(y(s))]\mathrm{d}E_s.
	\end{equation*}
	Applying the time-changed It\^{o} formula to $F(y(s))$ and $G(y(s))$ and using the Fubini theorem to exchange the order of integration, we decompose the remainder as
	\begin{equation*}
		\mathcal R_{k+1}=A_{k+1}+M_{k+1},
	\end{equation*}
	where the $A_{k+1} = A_{k+1}^{tt} + A_{k+1}^{Et} + A_{k+1}^{tE} + A_{k+1}^{EE}$ are given by
	\begin{align*}
		A_{k+1}^{tt} &= -\int_{t_k}^{t_{k+1}}(t_{k+1}-r)F'(y(r))F(y(r))\mathrm{d}r,\\
		A_{k+1}^{Et} &= -\int_{t_k}^{t_{k+1}}(t_{k+1}-r)\left[F'(y(r))G(y(r))+\frac{q^2}{2}F''(y(r))\right]\mathrm{d}E_r,\\
		A_{k+1}^{tE} &= -\int_{t_k}^{t_{k+1}}(E_{t_{k+1}}-E_r)G'(y(r))F(y(r))\mathrm{d}r,\\
		A_{k+1}^{EE} &= -\int_{t_k}^{t_{k+1}}(E_{t_{k+1}}-E_r)\left[G'(y(r))G(y(r))+\frac{q^2}{2}G''(y(r))\right]\mathrm{d}E_r,
	\end{align*}
	and the martingale components $M_{k+1} = M_{k+1}^{Bt} + M_{k+1}^{BE}$ are defined by
	\begin{align*}
		M_{k+1}^{Bt} &= -q\int_{t_k}^{t_{k+1}}(t_{k+1}-r)F'(y(r))\mathrm{d}B(E_r),\\
		M_{k+1}^{BE} &= -q\int_{t_k}^{t_{k+1}}(E_{t_{k+1}}-E_r)G'(y(r))\mathrm{d}B(E_r).
	\end{align*}
	Clearly, $M_{k+1}$ is a martingale increment satisfying $\mathbb E_B[M_{k+1}\mid\mathcal F_{t_k}]=0$.
	
	Define
	\begin{equation*}
		e_k=X_k-y_k,\qquad d_k=e^{X_k}-e^{y_k}.
	\end{equation*}
	Since $0\leq G'(x)\leq\frac{q^2}{4}$, the mean value theorem gives
	\begin{equation*}
		G(X_k)-G(y_k)=\theta_ke_k,\qquad0\leq\theta_k\leq\frac{q^2}{4}.
	\end{equation*}
	Moreover, by the definition of $F$,
	\begin{equation*}
		F(X_k)-F(y_k)=-(\mu+\gamma)d_k.
	\end{equation*}
	By comparing \eqref{eq:numerical-scheme2} with
	\eqref{eq:diff-exact-recursion}, we obtain
	\begin{equation}\label{eq:diff-error-recursion}
		e_{k+1}=Z_k+A_{k+1}+M_{k+1},
	\end{equation}
	where
	\begin{equation*}
		Z_k=(1+\theta_k\Delta E_k)e_k-(\mu+\gamma)d_k\Delta.
	\end{equation*}
	
	The monotonicity of the exponential function implies $e_kd_k\geq0$.
	For nonnegative $u$ and $v$, the inequality $|u-v|^p\leq u^p+v^p$ holds.
	Therefore,
	\begin{equation*}
		|Z_k|^p\leq\left(1+\frac{q^2}{4}\Delta E_k\right)^p|e_k|^p+C\Delta^p|d_k|^p.
	\end{equation*}
	For $x,y\in\mathbb R$, we have
	\begin{equation*}
		\frac{|e^x-e^y|}{e^x+e^y}=\left|\tanh\left(\frac{x-y}{2}\right)\right|\leq1\wedge|x-y|.
	\end{equation*}
	Since $(1\wedge a)^p\leq a$ for $a\geq0$, it follows that
	\begin{equation*}
		|e^x-e^y|^p\leq(e^x+e^y)^p|x-y|.
	\end{equation*}
	Applying this inequality and Young's inequality, we obtain
	\begin{align*}
		\Delta^p|d_k|^p&\leq\Delta^p(e^{X_k}+e^{y_k})^p|e_k|\\
		&\leq\frac{\Delta}{p}|e_k|^p+C\Delta^{p+1}(e^{X_k}+e^{y_k})^{p^2/(p-1)}.
	\end{align*}
	Consequently, Lemma \ref{lem:exp-moment-bounds} and
	$(1+q^2\Delta E_k/4)^p\leq e^{pq^2\Delta E_k/4}$ yield
	\begin{equation}\label{eq:diff-Z-bound}
		\mathbb E_B\left[|Z_k|^p\right]\leq e^{Ch_k}\mathbb E_B\left[|e_k|^p\right]+C\Delta^{p+1}e^{C(T+E_T)}.
	\end{equation}
	
	We next estimate the $L^p$ error from \eqref{eq:diff-error-recursion}.
	For any $u,v\in\mathbb R$ and integer $p\geq2$, Taylor's formula implies
	\begin{equation*}
		|u+v|^p \leq |u|^p + p|u|^{p-2}uv + C\bigl(|u|^{p-2}|v|^2 + |v|^p\bigr).
	\end{equation*}
	Applying this inequality with $u=Z_k$ and $v=A_{k+1}+M_{k+1}$, we obtain the explicit expansion
	\begin{align}\label{eq:diff-taylor-expansion}
		|e_{k+1}|^p &\leq |Z_k|^p + p|Z_k|^{p-2}Z_k M_{k+1} + p|Z_k|^{p-2}Z_k A_{k+1} \notag\\
		&\quad + C|Z_k|^{p-2}|A_{k+1}+M_{k+1}|^2 + C|A_{k+1}+M_{k+1}|^p.
	\end{align}
	We now estimate each term on the right-hand side of \eqref{eq:diff-taylor-expansion}.
	For the martingale term, since $Z_k$ is measurable at time $t_k$ and $M_{k+1}$ is a martingale increment, we obtain
	\begin{equation}\label{eq:diff-martingale-cross}
		\mathbb E_B\left[|Z_k|^{p-2}Z_kM_{k+1}\right]=0.
	\end{equation}
	For the first-order term, by Young's inequality, we have
	\begin{equation}\label{eq:diff-first-order-A}
		p|Z_k|^{p-2}Z_k A_{k+1} \leq p|Z_k|^{p-1}|A_{k+1}| \leq h_k|Z_k|^p + C h_k^{1-p}|A_{k+1}|^p.
	\end{equation}
	For the second-order terms, using $|A_{k+1}+M_{k+1}|^2 \leq 2|A_{k+1}|^2 + 2|M_{k+1}|^2$ and Young's inequality, together with $h_k^{p/2} \leq C e^{Ch_k}$, we obtain
	\begin{align}
		|Z_k|^{p-2}|A_{k+1}|^2 &\leq \frac{h_k}{4}|Z_k|^p + C h_k^{-\frac{p-2}{2}}|A_{k+1}|^p \leq \frac{h_k}{4}|Z_k|^p + C e^{Ch_k}h_k^{1-p}|A_{k+1}|^p,\label{eq:diff-second-order-cross1}\\
		|Z_k|^{p-2}|M_{k+1}|^2 &\leq \frac{h_k}{4}|Z_k|^p + C h_k^{1-p/2}|M_{k+1}|^p. \label{eq:diff-second-order-cross2}
	\end{align}
	For the higher-order remainder, since $h_k^r \leq C_r e^{Ch_k}$ for every $r>0$, we have
	\begin{align}\label{eq:diff-higher-order}
		|A_{k+1}+M_{k+1}|^p &\leq C\bigl(|A_{k+1}|^p+|M_{k+1}|^p\bigr) \leq C e^{Ch_k}\left(h_k^{1-p}|A_{k+1}|^p + h_k^{1-p/2}|M_{k+1}|^p\right).
	\end{align}
	Taking the conditional expectation $\mathbb E_B$ on both sides of \eqref{eq:diff-taylor-expansion}, combining estimates \eqref{eq:diff-martingale-cross}--\eqref{eq:diff-higher-order}, and noting that $1+Ch_k \leq e^{Ch_k}$, we arrive at
	\begin{align}\label{eq:diff-one-step-error}
		\mathbb E_B\left[|e_{k+1}|^p\right] &\leq e^{Ch_k}\left(\mathbb E_B\left[|Z_k|^p\right] + C\mathbb E_B\left[h_k^{1-p}|A_{k+1}|^p + h_k^{1-p/2}|M_{k+1}|^p\right]\right)\nonumber\\
		&\leq e^{Ch_k}\Bigg( \mathbb E_B\left[|e_k|^p\right] + C\Delta^{p+1}e^{C(T+E_T)} + C\mathbb E_B\left[h_k^{1-p}|A_{k+1}|^p + h_k^{1-p/2}|M_{k+1}|^p\right] \Bigg),
	\end{align}
	where we use \eqref{eq:diff-Z-bound} in the last inequality.
	
	It remains to estimate the residual terms on the right-hand side of \eqref{eq:diff-one-step-error}.
	By Lemma \ref{lem:exp-moment-bounds} and conditions \eqref{eq:F-growth-bounds}--\eqref{eq:G-growth-bounds}, the integrands in $A_{k+1}$ and $M_{k+1}$ satisfy the uniform moment bound
	\begin{align}\label{eq:diff-integrand-bound}
		\sup_{0\leq r\leq T}& \mathbb E_B\Bigg[
		|F'(y(r))F(y(r))|^p + \left|F'(y(r))G(y(r))+\frac{q^2}{2}F''(y(r))\right|^p + |G'(y(r))F(y(r))|^p \notag\\
		&+ \left|G'(y(r))G(y(r))+\frac{q^2}{2}G''(y(r))\right|^p + |F'(y(r))|^p + |G'(y(r))|^p \Bigg] \leq Ce^{C(T+E_T)}.
	\end{align}
	
	We next estimate the four components of $A_{k+1}$ as follows.
	For $A_{k+1}^{tt}$, using H\"{o}lder's inequality along with $h_k \geq \Delta$, we have
	\begin{equation*}
		h_k^{1-p}\mathbb E_B\left[|A_{k+1}^{tt}|^p\right] \leq \Delta^p\int_{t_k}^{t_{k+1}}\mathbb E_B\left[|F'(y(r))F(y(r))|^p\right]\mathrm{d}r \leq C\Delta^p e^{C(T+E_T)}\Delta.
	\end{equation*}
	Similarly, for $A_{k+1}^{Et}$ and $A_{k+1}^{tE}$, using $h_k\geq\Delta E_k$, we obtain
	\begin{align*}
		h_k^{1-p}\mathbb E_B\left[|A_{k+1}^{Et}|^p\right] &\leq \Delta^p\int_{t_k}^{t_{k+1}}\mathbb E_B\left[\left|F'(y(r))G(y(r))+\frac{q^2}{2}F''(y(r))\right|^p\right]\mathrm{d}E_r \\
		&\leq C\Delta^p e^{C(T+E_T)}\Delta E_k\\
		h_k^{1-p}\mathbb E_B\left[|A_{k+1}^{tE}|^p\right] &\leq \Delta E_k\Delta^{p-1}\int_{t_k}^{t_{k+1}}\mathbb E_B\left[|G'(y(r))F(y(r))|^p\right]\mathrm{d}r \\
		&\leq C\Delta^p e^{C(T+E_T)}\Delta E_k.
	\end{align*}
	For $A_{k+1}^{EE}$, since the compound coefficient of $G$ is uniformly bounded by \eqref{eq:G-growth-bounds} and $h_k \geq \Delta E_k$,
	\begin{equation*}
		h_k^{1-p}|A_{k+1}^{EE}|^p \leq \frac{C}{h_k^{p-1}}\left[\int_{t_k}^{t_{k+1}}(E_{t_{k+1}}-E_r)\mathrm{d}E_r\right]^p \leq C(\Delta E_k)^{p+1}.
	\end{equation*}
	Using $|A_{k+1}|^p \leq 4^{p-1}\bigl(|A_{k+1}^{tt}|^p + |A_{k+1}^{Et}|^p + |A_{k+1}^{tE}|^p + |A_{k+1}^{EE}|^p\bigr)$ and summing over $k=0,\dots,n_\Delta-1$, we obtain
	\begin{align}\label{eq:diff-A-total-bound}
		\sum_{k=0}^{n_\Delta-1}h_k^{1-p}\mathbb E_B\left[|A_{k+1}|^p\right] &\leq Ce^{C(T+E_T)}\sum_{k=0}^{n_\Delta-1}\left(\Delta^{p+1} + \Delta^p\Delta E_k\right) + C\sum_{k=0}^{n_\Delta-1}(\Delta E_k)^{p+1} \notag\\
		&\leq Ce^{C(T+E_T)}\left(\Delta^p + \sum_{k=0}^{n_\Delta-1}(\Delta E_k)^{p+1}\right).
	\end{align}
	We now estimate the martingale terms.
	For $M_{k+1}^{Bt}$, applying the Burkholder--Davis--Gundy (BDG) inequality and H\"{o}lder's inequality with respect to the measure $\mathrm{d}E_r$ yields
	\begin{align*}
		\mathbb E_B\left[|M_{k+1}^{Bt}|^p\right] &\leq C\Delta^p\mathbb E_B\left[\left(\int_{t_k}^{t_{k+1}}|F'(y(r))|^2\mathrm{d}E_r\right)^{p/2}\right] \\
		&\leq C\Delta^p(\Delta E_k)^{\frac{p-2}{2}}\int_{t_k}^{t_{k+1}}\mathbb E_B\left[|F'(y(r))|^p\right]\mathrm{d}E_r.
	\end{align*}
	Since $h_k \geq \Delta E_k$, dividing by $h_k^{\frac{p-2}{2}}$ gives
	\begin{equation*}
		h_k^{1-p/2}\mathbb E_B\left[|M_{k+1}^{Bt}|^p\right] \leq C\Delta^p\int_{t_k}^{t_{k+1}}\mathbb E_B\left[|F'(y(r))|^p\right]\mathrm{d}E_r \leq C\Delta^p e^{C(T+E_T)}\Delta E_k.
	\end{equation*}
	For $M_{k+1}^{BE}$, using the uniform bound $\|G'\|_\infty \leq \frac{q^2}{4}$ and $h_k \geq \Delta E_k$, the BDG inequality gives
	\begin{align*}
		h_k^{1-p/2}\mathbb E_B\left[|M_{k+1}^{BE}|^p\right] &\leq \frac{C}{h_k^{\frac{p-2}{2}}}\left[\int_{t_k}^{t_{k+1}}(E_{t_{k+1}}-E_r)^2\mathrm{d}E_r\right]^{p/2} \\
		&\leq C\frac{(\Delta E_k)^{\frac{3p}{2}}}{(\Delta E_k)^{\frac{p-2}{2}}} = C(\Delta E_k)^{p+1}.
	\end{align*}
	Using $|M_{k+1}|^p \leq 2^{p-1}\bigl(|M_{k+1}^{Bt}|^p + |M_{k+1}^{BE}|^p\bigr)$ and summing over $k=0,\dots,n_\Delta-1$, we deduce that
	\begin{align}\label{eq:diff-M-total-bound}
		\sum_{k=0}^{n_\Delta-1}h_k^{1-p/2}\mathbb E_B\left[|M_{k+1}|^p\right] &\leq Ce^{C(T+E_T)}\Delta^p\sum_{k=0}^{n_\Delta-1}\Delta E_k + C\sum_{k=0}^{n_\Delta-1}(\Delta E_k)^{p+1}\notag\\
		&\leq Ce^{C(T+E_T)}\left(\Delta^p + \sum_{k=0}^{n_\Delta-1}(\Delta E_k)^{p+1}\right).
	\end{align}
	Combining \eqref{eq:diff-A-total-bound} and \eqref{eq:diff-M-total-bound}, we conclude that
	\begin{align}\label{eq:diff-residual-bound}
		\sum_{k=0}^{n_\Delta-1}h_k^{1-p}&\mathbb E_B\left[|A_{k+1}|^p\right] + \sum_{k=0}^{n_\Delta-1}h_k^{1-p/2}\mathbb E_B\left[|M_{k+1}|^p\right] \notag\\
		&\leq Ce^{C(T+E_T)}\left(\Delta^p + \sum_{k=0}^{n_\Delta-1}(\Delta E_k)^{p+1}\right).
	\end{align}
	
	Iterating \eqref{eq:diff-one-step-error} from $k=0$ to $j-1$ yields
	\begin{align*}
		\mathbb E_B\left[|e_j|^p\right] &\leq Ce^{C(T+E_T)}\Bigg\{ \sum_{k=0}^{j-1}\Delta^{p+1} + \sum_{k=0}^{j-1}\mathbb E_B\left[h_k^{1-p}|A_{k+1}|^p + h_k^{1-p/2}|M_{k+1}|^p\right] \Bigg\} \notag\\
		&\leq Ce^{C(T+E_T)}\left(\Delta^p + \sum_{k=0}^{n_\Delta-1}(\Delta E_k)^{p+1}\right),
	\end{align*}
	where we have used $\sum_{k=0}^{j-1}h_k \leq (T+E_T)$ and \eqref{eq:diff-residual-bound} in the last inequality.
	Taking $\mathbb E_D$ on both sides and taking the maximum over $0\leq j\leq n_\Delta$, applying Lemma \ref{lemma:29} and Lemma \ref{lem:sharp-clock}, we obtain
	\begin{align*}
		\max_{0\leq j\leq n_\Delta}\mathbb E\left[|e_j|^p\right]&\leq C\Delta^p\mathbb E\left[e^{CE_T}\right]+C\mathbb E\left[e^{CE_T}\sum_{k=0}^{n_\Delta-1}(\Delta E_k)^{p+1}\right]\\
		&\leq C\left(\Delta^p+\Delta^{p\alpha}\right)\\
		&\leq C\Delta^{p\alpha}.
	\end{align*}
	This establishes the first convergence estimate.
	Next, recall that the inverse transformation $\varphi(x) = \frac{Ne^x}{1+e^x}$ satisfies $\|\varphi'\|_\infty \leq N/4$.
	Since $\varphi$ is globally Lipschitz continuous, taking the maximum and the expectation directly yields
	\begin{equation*}
		\max_{0\leq k\leq n_\Delta}\mathbb E\left[|I_k - I(t_k)|^p\right] \leq \left(\frac{N}{4}\right)^p \max_{0\leq k\leq n_\Delta}\mathbb E\left[|X_k - y(t_k)|^p\right] \leq C\Delta^{p\alpha}.
	\end{equation*}
	This completes the proof.
\end{proof}
\section{Numerical simulations}
\label{sec:numerical_simulation}

In this section, we provide numerical experiments to verify the strong convergence order established in Theorem \ref{thm:strong-conv-diffusion}.

We first simulate the inverse stable subordinator $E_t$ using the duality construction \cite{DengLiu2020,Jum2016}.
For an operational-time step size $\delta>0$ and physical-time horizon $T>0$, we generate $D_0=0$ and $D_{i\delta}=D_{(i-1)\delta}+Z_i$ for $i\geq1$, where $Z_i\stackrel{d}{=}D_\delta$ are i.i.d. stable random variables.
The procedure stops at $N_T$ once $T\in[D_{N_T\delta},D_{(N_T+1)\delta})$.
The discrete inverse subordinator is defined by $E_t^\delta = n_t\delta$ with $n_t = \max\{n\geq0 : D_{n\delta}\leq t\}$.

The model parameters are set to $N=100$, $\beta=0.5$, $\mu=20$, $\gamma=25$, $\sigma=0.035$, and $I_0=90$.
The time horizon is $T=2$. Figure \ref{fig:Trajectories} presents several sample paths generated by the LEM method with step size $\Delta=10^{-3}$ for $\alpha=0.6$ and $\alpha=0.8$, respectively. It can be observed that all numerical trajectories remain in the biologically meaningful interval $(0,N)$, illustrating the positivity-preserving property of the proposed method. 

To further verify the strong convergence result, expectations are approximated via $M=100$ independent sample paths over the multi-scale step sizes
\begin{equation*}
	\Delta \in \{2^{-16}, 2^{-17}, 2^{-18}, 2^{-19}, 2^{-20}\}.
\end{equation*}
The numerical solution generated with the step size $\Delta=2^{-20}$ is used as the reference solution.
We apply the explicit scheme \eqref{eq:numerical-scheme2} directly on the uniform physical-time grid.
Figure \ref{fig:model2_convergence} presents the corresponding root mean square errors for $\alpha=0.6$ and $\alpha=0.8$.
The error curves match the reference lines of slopes $0.6$ and $0.8$, respectively, in full agreement with the strong convergence order $\alpha$ proved in Theorem \ref{thm:strong-conv-diffusion}.

\begin{figure}[htbp]
	\centering
	\begin{subfigure}{0.49\textwidth}
		\centering
		\includegraphics[width=\textwidth]{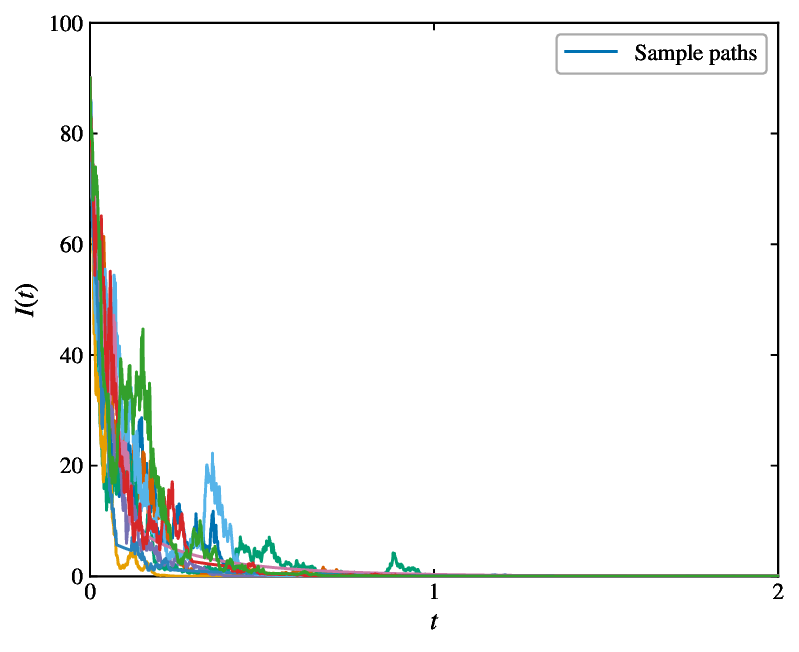}
		\caption{$\alpha=0.6$ }
		\label{fig:Xt_alpha06}
	\end{subfigure}
	\hfill
	\begin{subfigure}{0.49\textwidth}
		\centering
		\includegraphics[width=\textwidth]{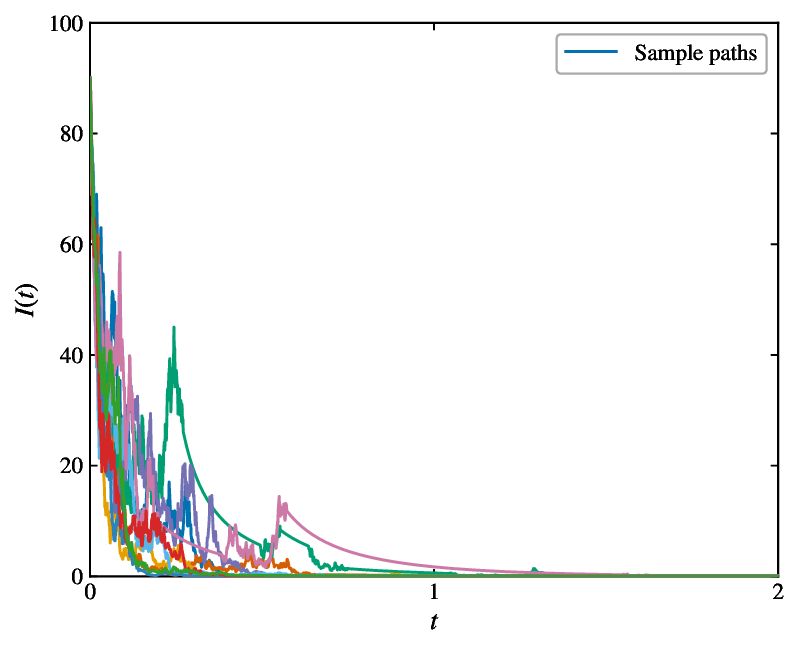}
		\caption{$\alpha=0.8$}
		\label{fig:Xt_alpha08}
	\end{subfigure}
	\caption{Sample trajectories generated by the LEM method with step size $\Delta=10^{-3}$.}
	\label{fig:Trajectories}
\end{figure}

\begin{figure}[htbp]
	\centering
	\begin{subfigure}{0.49\textwidth}
		\centering
		\includegraphics[width=\textwidth]{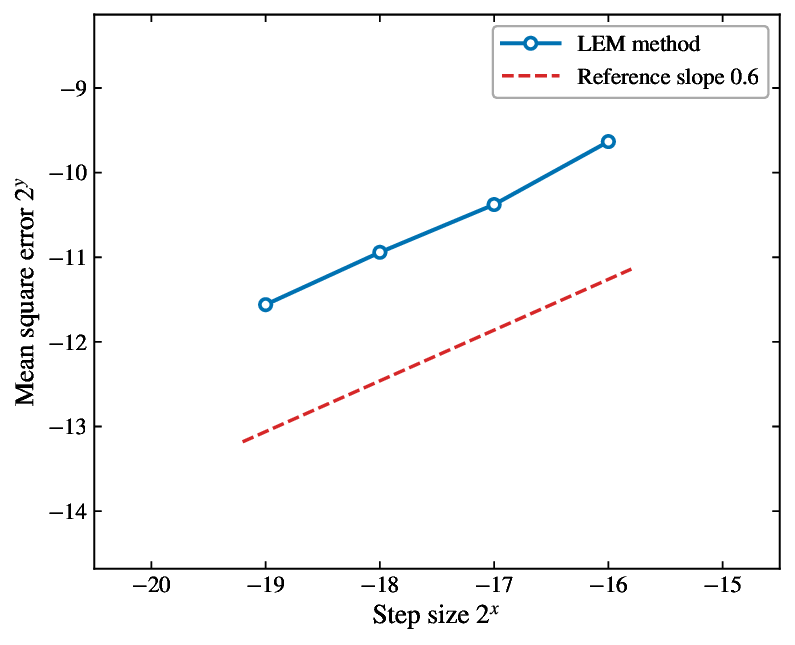}
		\caption{$\alpha=0.6$ (reference slope $0.6$)}
		\label{fig:model2_alpha06}
	\end{subfigure}
	\hfill
	\begin{subfigure}{0.49\textwidth}
		\centering
		\includegraphics[width=\textwidth]{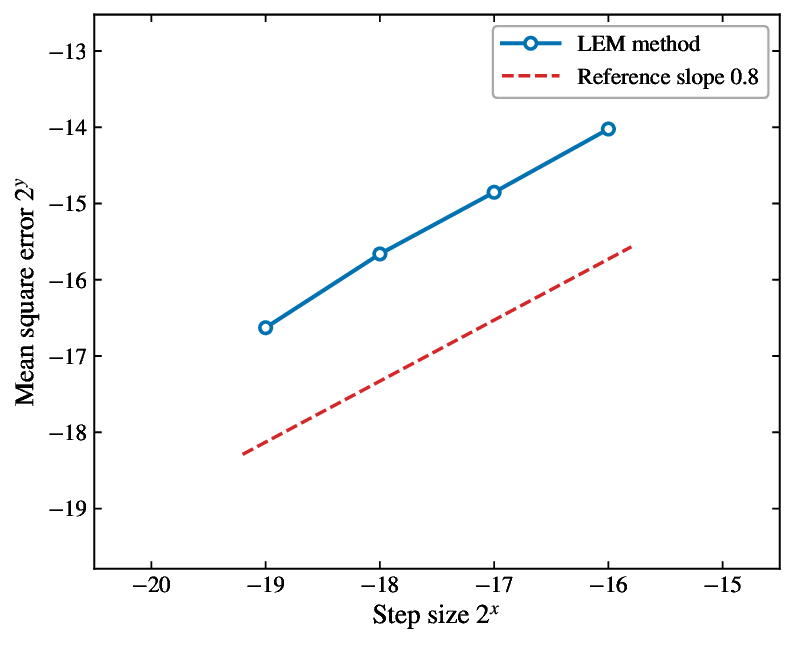}
		\caption{$\alpha=0.8$ (reference slope $0.8$)}
		\label{fig:model2_alpha08}
	\end{subfigure}
	\caption{Strong convergence order for the diffusion time-changed SIS model.}
	\label{fig:model2_convergence}
\end{figure}

\section{Conclusion}\label{sec:conclusion}

The numerical experiments lead to the following observations.

First, the proposed LEM scheme preserves positivity in all numerical simulations. In all simulations, the numerical approximations remain in the interval $(0,N)$. This is attributed to the Lamperti transform, which maps the original state space $(0, N)$ to the entire real axis, ensuring solution positivity by algorithmic design.

Second, the convergence analysis results are highly consistent with theoretical predictions. 
For the diffusion time-changed  SIS model \eqref{eq:model33}, the convergence order of the LEM scheme is precisely characterized by the stability index $\alpha$ of the inverse subordinator: when $\alpha=0.6$, the convergence order is approximately $0.6$; when $\alpha=0.8$, the convergence order is approximately $0.8$. This intuitively reflects the impact of the path irregularity of the time-change process $E_t$ on the numerical accuracy mechanism.

Finally, the numerical simulations not only validate the correctness of the theoretical analysis presented in this paper but also reveal the significant influence of the time-change parameter $\alpha$ on numerical computations in time-changed stochastic SIS models. These results support the effectiveness of the proposed LEM method for the numerical approximation of diffusion time changed stochastic SIS models.

\bibliographystyle{abbrv}
\bibliography{myref}

\end{document}